\documentclass[dvips]
{article}
\usepackage[a5paper,layoutwidth=148.5mm,layoutheight=7.77817in,
outer=10mm,inner=38pt,tmargin=10mm,bmargin=10mm,footskip=15pt,
]{geometry}

\usepackage[T1]{fontenc}
\usepackage{lmodern}
\usepackage{fix-cm}
\usepackage{textcomp}
\usepackage{microtype}

\usepackage[english]{babel}

\usepackage{pst-eucl,pst-plot}
\makeatletter
\def\pstAbscissa#1{%
  tx@EcldDict begin /N@#1 GetNode pop \pst@number\psxunit div end
}%
\def\pstOrdinate#1{%
  tx@EcldDict begin /N@#1 GetNode exch pop \pst@number\psyunit div end
}%
\def\pstShowCoor#1{
\begin@ClosedObj
  \addto@pscode{%
    tx@EcldDict begin /N@#1 GetNode exch \pst@number\psyunit div = \pst@number\psyunit div = end%
  }
\end@ClosedObj
}%
\def\pstMoveNode{\@ifnextchar[\Pst@MoveNode{\Pst@MoveNode[]}}
\def\Pst@MoveNode[#1]{%
  \begingroup
    \psset{#1}%
    \Pst@MoveNode@i}
\def\Pst@MoveNode@i(#1,#2)#3#4{%
  \pnode(! \pstAbscissa{#3} #1 add \pstOrdinate{#3} #2 add){#4}%
  \Pst@geonodelabel{#4}%
  \endgroup%
}%
\makeatother

\usepackage[matrix,arrow,curve]{xy}

\usepackage[shortlabels]{enumitem}
\setlist{nosep}

\usepackage[numbers]{natbib}
\usepackage{amsmath}
\usepackage{amsfonts}
\usepackage{mathtools} \mathtoolsset{mathic} 

\RequirePackage{etoolbox}
\usepackage{amsthm}
\makeatletter\def\th@plain{\slshape}\makeatother
\makeatletter\patchcmd{\th@remark}{\itshape}{\slshape}{}{}\makeatother
\AtBeginDocument{}
\theoremstyle{plain}
\newtheorem{theorem}{Theorem}
\newtheorem{proposition}[theorem]{Proposition}
\theoremstyle{definition}
\newtheorem{definition}[theorem]{Definition}

\theoremstyle{remark}
\newtheorem{remark}[theorem]{Remark}

\usepackage{indentfirst}
\PassOptionsToPackage{urlcolor=magenta,citecolor=blue,anchorcolor=blue,linkcolor=blue,menucolor=cyan}{hyperref}%
\makeatletter
\ifpdf
  \expandafter\@firstofone
\else
  \expandafter\@gobble
\fi
{\usepackage[pdftitle=,pdfauthor=,pdfdisplaydoctitle=true,bookmarksopen=false,unicode=true,plainpages=false,hyperindex=true,pdfstartview=FitH,colorlinks=true,pdfpagelabels=true,bookmarksnumbered=true]{hyperref}\usepackage{doi}\usepackage{breakurl}}

\newcommand\bA{\mathrm{A}}

\newcommand\bS{\mathrm{S}}
\newcommand\bH{\mathrm{H}}
\newcommand\bR{\mathrm{R}}

\newcommand\bK{\mathrm{K}}
\newcommand\bV{\mathrm{V}}
\newcommand\bValpha{\bV_{\!\alpha}}
\newcommand\bL{\mathrm{L}}
\newcommand\bW{\mathrm{W}}
\newcommand\rh{\mathrm{h}}
\newcommand\bRh{\bR^{\mspace{-2mu}\rh}}
\newcommand\bAtimes{\bA^{\mspace{-3mu}\times}}
\newcommand\bRtimes{\bR^{\mspace{-2mu}\times}}
\newcommand\bKtimes{\bK^{\mspace{-2mu}\times}}
\newcommand\bVtimes{\bV^{\mspace{-2mu}\times}}
\newcommand\bVf{\bV_{\!f}}
\newcommand\rH{\mathrm{H}}
\newcommand\fp[1][]{\mathfrak p_{#1}}
\newcommand\im[1][]{\mathfrak m_{#1}}
\newcommand\cre[1][]{\bK_{#1}}
\newcommand\gen[1]{\langle #1\rangle}
\newcommand\vg{\Gamma}
\newcommand\vgi{\vg_{\!\infty}}

\begin{document}

\title{\vspace*{-15mm}On a theorem by de Felipe and Teissier about\\the comparison of two henselisations\\in the non-noetherian case}
\author
{María Emilia Alonso García
\and Henri Lombardi
\and Stefan Neuwirth
}
\date{\vspace*{-7mm}}

\maketitle

\date{}

\begin{abstract}
  Let \(\bR\)~be a local domain, \(v\)~a valuation of its quotient field centred in~\(\bR\) at its maximal ideal. We investigate the relationship between~\(\bR^\rh\), the henselisation of~\(\bR\) as local ring, and \(\tilde{v}\), the henselisation of the valuation~\(v\), by focussing on the recent result by de Felipe and Teissier referred to in the title.  We give a new proof that simplifies the original one by using purely algebraic arguments. This proof is moreover constructive in the sense of Bishop and previous work of the authors, and allows us to obtain as a by-product a (slight) generalisation of the theorem by de Felipe and Teissier.

  
\end{abstract}

\section{Introduction}
\label{sec:introduction}

The theorem referred to in the title is the following, published in~\cite{felipeteissier20}.
\begin{theorem} \label{thFT}
Let \(\bR\) be a local domain and let \(\bR^\rh\) be its henselisation. If \(v\)~is a valuation centred in~\(\bR\), then:
\begin{enumerate}[(1)]
\item \label{thFT1} There exists a unique prime ideal~\(\bH(v)\) of~\(\bR^\rh\) lying over the zero ideal of~\(\bR\) such that \(v\)~extends to a valuation \(\tilde{v}\) centred in~\(\bR^\rh/\bH(v)\) through the inclusion \(\bR\subseteq\bR^\rh/\bH(v)\). In addition, the ideal \(\bH(v)\) is a minimal prime and the extension \(\tilde{v}\) is unique.
\item \label{thFT2} With the notation of~\ref{thFT1}, the valuations \(v\)~and~\(\tilde{v}\) have the same value group.
\end{enumerate}
\end{theorem}

This theorem compares in fact two henselisations: of the local domain~\(\bR\) and of the valuation ring defined by~\(v\). Item~\ref{thFT1} says that the canonical local morphism of the first henselisation into the second has a kernel that is a minimal prime ideal. Item~\ref{thFT2} is according to us a plain reminder that henselising a valuation ring leaves the value group unchanged.

The proof technique developed by the authors of~\cite{felipeteissier20} is quite sophisticated and uses topological methods, e.g.\ pseudo-convergent series in ad hoc completions. Yet all the concepts appearing in the theorem allow a basic algebraic definition. For instance, the notions of valuation domain, of henselisation of a local ring, and of henselisation of a valuation domain can be described in basic algebra. A purely algebraic proof of the theorem at stake is thus a priori conceivable. This is what is obtained in this article. An additional benefit is that our proofs are entirely constructive when the hypotheses are expressed with enough care. Our method proves in fact a little bit more, viz.
\begin{theorem}
  Let \(\bR\) be a local ring, \(\bV\) a valuation domain, and \(\phi\colon\bR\to\bV\) a local morphism whose kernel is a minimal prime ideal. Let \(\bR^\rh\) be the henselisation of~\(\bR\) as a local ring, and let \(\bV^\rH\) be the henselisation of~\(\bV\) as a valuation domain. The kernel of the canonical local morphism \(\psi\colon\bR^\rh\to\bV^\rH\) is a minimal prime ideal.
\end{theorem}
In fact, this theorem holds even true at every step of the simultaneous construction of the henselisations, as shown in Theorem~\ref{th-step}. Note that henselising a local domain does not in general yield a domain: this explains why we had to prove a little bit more for proceeding step by step; see Remarks~\ref{rem-reduit} and~\ref{rem-single} for two variants.

\section{Setting}
\label{sec:setting}

By tuning our definitions as follows, we are able to make the whole article constructive, in agreement with the stance of \cite{Bi67,MRR}.

A \textsl{discrete field} is a nontrivial commutative ring in which every element is zero or invertible.
A \textsl{local ring} is a nontrivial commutative ring with unit such that \[x+y\in\bRtimes\implies{x\in\bRtimes}\text{ or }{y\in\bRtimes}\text.\]
A ring morphism \(\phi\colon\bR\to\bA\) between local rings is \textsl{local} if it reflects units, i.e.\ if \[\phi(x)\in\bAtimes\implies x\in\bRtimes\text.\]
The \textsl{Jacobson radical} of a commutative ring~\(\bR\) is \(\operatorname{Rad}(\bR)=\{x:1+x\bR\subseteq\bRtimes\}\). When the ring is local, \(\operatorname{Rad}(\bR)\) is a maximal prime ideal that we shall denote as usual by \(\im[\bR]\);
the field \(\cre[\bR]=\bR/\im[\bR]\) is the \textsl{residue field} of~\(\bR\); \(\bR\) is \textsl{residually discrete} if \(\cre[\bR]\) is a discrete field.

A subring~\(\bV\) of a discrete field~\(\bK\) is a \textsl{valuation ring} for~\(\bK\) if \[x\in\bKtimes\implies {x\in\bV}\text{ or }x^{-1}\in\bV\text;\]
in particular, \(\bV\) is a local ring. A valuated discrete field is a pair~\((\bK,\bV)\) with \(\bK\)~a discrete field and \(\bV\) a valuation ring for~\(\bK\). Its \textsl{value group} is the linearly ordered group~\(\vg=\bKtimes/\bVtimes\), and its \textsl{valuation} \(v\colon\bK\to\vgi=\vg\cup\{\infty\}\) is defined as the extension of the canonical morphism \(\bKtimes\to\vg\) by setting \(v(0)=\infty\). If \(\bV\) is residually discrete, then \(\vg\)~is discrete. 

A ring is \textsl{zero-dimensional} if for each element~\(x\) there is an element~\(y\) and an integer~\(N\ge0\) such that \(x^N(1+xy)=0\). A prime ideal~\(\fp\) in a ring~\(\bR\) is \textsl{minimal} if the localisation of~\(\bR\) at~\(\fp\) is zero-dimensional.

\section{The henselisation of a residually discrete local ring}
\label{sec:hens-resid-discr}

\begin{definition}
  Let \((\bR,\im[\bR])\) be a local ring with residue field~\(\cre[\bR]=\bR/\im[\bR]\).
  \begin{itemize}
  \item An \textsl{henselian zero above}~\(x\in\bR\) for a polynomial $f\in\bR[X]$ subject to the conditions \(\overline{f(x)}=0\) in~\(\cre[\bR]\) and \(\overline{f'(x)}\in(\cre[\bR])^\times\) is an $\alpha\in x+\im[\bR]$ such that \(f(\alpha)=0\). An \textsl{henselian zero} is an henselian zero above~\(0\).
  \item\sloppy A \textsl{Nagata polynomial} in~\(\bR[X]\) is a monic
    polynomial \(f(X)=X^n+a_{n-1}X^{n-1}+\dots+a_1X+a_0\) with
    \(\overline{a_1}\in(\cre[\bR])^\times\) and \(\overline{a_0}=0\) in~\(\cre[\bR]\).
  \item A \textsl{special polynomial} in~\(\bR[X]\) is a monic
    polynomial \(t(X)=X^n-X^{n-1}+t_{n-2}X^{n-2}+\dots+t_0\) with the~\(t_i\) in~\(\im[\bR]\).
  \item \(\bR\) is \textsl{henselian} if every Nagata polynomial has an henselian zero.
  \item An \textsl{henselisation of}~\(\bR\) is an henselian local ring \((\bRh,\im[\bRh])\) together with a local morphism \(\phi^\rh\colon\bR\to\bRh\) that factorises in a unique way every local morphism of~\(\bR\) to an henselian local ring.
  \end{itemize}
\end{definition}
\begin{proposition}
  There is at most one henselian zero above a given element for a given polynomial. A special polynomial has at most one zero~\(\beta\) with \(\overline{\beta}=1\) in~\(\cre[\bR]\), called its \textup{special zero}; if the special polynomial has other zeros in \(\bR\), they are in~\(\im[\bR]\).
\end{proposition}
When \(\bR\) is residually discrete, then the henselisation of~\(\bR\) exists and may be constructed as filtered colimit of successive extensions of~\(\bR\) by henselian zeros of Nagata polynomials. This results from the following theorem.
\begin{theorem}[{\cite[Lemmas~6.2 and~6.3]{ALP08}}]\label{hensel-rdlr}
  Let \((\bR,\im[\bR])\) be a residually discrete local ring and \(f\) a Nagata polynomial in~\(\bR[X]\). Let \(\bR_f\) be the ring defined as the localisation of \(\bR[x]=\bR[X]/\gen f\) at the monoid
  \[\bS=\{\,g(x)\in\bR[x]:g\in\bR[X]\text{ and }\overline{g(0)}\in(\cre[\bR])^\times\,\}\text.\]
  \begin{enumerate}
  \item \(\bR_f\) is a residually discrete local ring with maximal ideal~\(\im[\bR]\bR_f\) and residue field canonically isomorphic to~\(\cre[\bR]\). The image of~\(x\) in~\(\bR_f\) is the henselian zero for the image of~\(f\) in~\(\bR_f[X]\). 
  \item \(\bR_f\) is faithfully  flat over~\(\bR\), so that \(\bR\)~may be identified with its image in~\(\bR_f\).
  \item\label{point3} The inclusion \(\bR\to\bR_f\) factorises in a unique way every local morphism of~\(\bR\) to a residually discrete local ring for which the image of~\(f\) has an henselian zero.
  \end{enumerate}
\end{theorem}
The universal property in Item~\ref{point3} is exactly what is needed to construct the henselisation as filtered colimit of the \(\bR_f\)'s: compare how \cite[Lemma~6.3]{ALP08} is used in \cite[Section~6.1.2]{ALP08}. Note that the image of~\(f\) in the residue field~\(\cre[\bR]\) has the henselian zero~\(0\): as Theorem~\ref{hensel-rdlr} aims in particular for a factorisation \(\bR\to\bR_f\to\cre[\bR]\) of local morphisms, the heuristics are that we must modify~\(\bR[X]/\gen f\) by making the elements of~\(\bS\) invertible; on the other hand, \(\bS\)~is maximal.

Note that the morphism \(\pi_f\colon\bR[x]\to\bR_f\) is in general not injective.
\section{The henselisation of a valuated discrete field}
\label{sec:hens-discr-valu}

\begin{definition}
  Let \((\bK,\bV)\) be a valuated discrete field such that \(\bV\)~is residually discrete.
  \begin{itemize}
  \item An \textsl{extension} of~\((\bK,\bV)\) is a valuated discrete field~\((\bL,\bW)\) with \(\bW\)~residually discrete together with a morphism~\(\phi\colon\bK\to\bL\) such that \(\bV=\bK\cap\phi^{-1}(\bW)\).
  \item \((\bK,\bV)\) is \textsl{henselian} if \(\bV\)~is henselian as a local ring.
  \item An \textsl{henselisation of}~\((\bK,\bV)\) is an henselian valuated discrete field \((\bK^\rH,\bV^\rH)\) together with a morphism \(\phi^\rH\colon\bK\to\bK^\rH\) that factorises in a unique way every extension of~\((\bK,\bV)\) to an henselian valuated discrete field.
  \end{itemize}
\end{definition}
The henselisation of~\((\bK,\bV)\) may be constructed as filtered colimit of successive extensions of~\((\bK,\bV)\) by henselian zeros of Nagata polynomials. This results from the following theorem, which is a version of the results of~\cite{KL00}.
\begin{theorem}\label{hensel-dvf}
  Let \((\bK,\bV)\) be a valuated discrete field such that \(\bV\)~is residually discrete and let \(f=X^n+a_{n-1}X^{n-1}+\dots+a_0\)~be a Nagata polynomial in~\(\bV[X]\). There is an extension~\((\bK[\alpha],\bValpha)\) of~\((\bK,\bV)\) for which the image of~\(f\) in~\(\bValpha[X]\) has an henselian zero~\(\alpha\) and such that the extension map \(\bK\to\bK[\alpha]\) factorises in a unique way every extension of~\((\bK,\bV)\) to~\((\bL,\bW)\) such that the image of~\(f\) in \(\bW[X]\) has an henselian zero. Furthermore, the residue field and the value group of~\((\bK[\alpha],\bValpha)\) are canonically isomorphic to the residue field and the value group of~\((\bK,\bV)\), respectively.
\end{theorem}

Before sketching a proof, let us recall the key theorem and the key proposition in~\cite{KL00}.

\begin{theorem}[{\cite[théorème~1]{KL00}}]\label{acvf}
  Let \(\mathbf{T}_{\mathrm{acvf}}\) be the formal theory of algebraically closed valuated fields based on the language of rings with furthermore a valuation ring membership predicate~\(\operatorname{Vr}(x)\). If \((\bK,\bV)\) is a valuated field, the formal theory~\(\mathbf{T}_{\mathrm{acvf}}(\bK,\bV)\) constructed from \(\mathbf{T}_{\mathrm{acvf}}\) by introducing as constants and relations the diagram of~\((\bK,\bV)\) is consistent.
\end{theorem}

One may therefore compute as if the valuated algebraic closure~\((\bK^{\mathrm{ac}},\bV^{\mathrm{ac}})\) of~\((\bK,\bV)\) existed.

\begin{proposition}[{\cite[proposition~2.2, Newton polygon version of Hensel's lemma]{KL00}}]\label{newton}
  Let \((\bK,\bV)\) be a valuated discrete field such that \(\bV\)~is residually discrete, and let \(g=b_0+b_{1}X+\dots+b_nX^n\)~be a polynomial in~\(\bV[X]\). Consider the \textup{Newton polygon} of~\(g\), i.e.\ the lower convex hull of the sequence of coordinates \((0,v(b_0)),\allowbreak(1,v(b_1)),\allowbreak\dots,(n,v(b_n))\) in the cartesian plane, defined formally as the subsequence in which the coordinates \((k,v(b_k))\) and~\((l,v(b_l))\) are consecutive if and only if the slope of the segment joining them is
  \begin{itemize}
  \item greater than the slopes from~\((i,v(b_i))\) to~\((k,v(b_k))\) for \(0\le i<k\),
  \item less than or equal to these slopes for \(k<i<l\),
  \item less than the slopes from~\((l,v(b_l))\) to~\((i,v(b_i))\) for \(l<i\le n\).
  \end{itemize}
  Suppose that \(k\)~is such that~\((k,v(b_k))\) and~\((k+1,v(b_{k+1}))\) are consecutive coordinates, i.e.\ that they correspond to an \textup{isolated slope} of the Newton polygon.
  \begin{figure}[!h]
    \centering
    \psset{xunit=1.9,yunit=.8,dotsize=7pt,arrowsize=3pt 2,PointNameSep=2.7em}
    \begin{pspicture}(-0.5,-3.25)(5.8,5)
      \pstGeonode[CurveType=polyline,PointName={{(0,v(b_0))},{(2,v(b_2))},{(3,v(b_3))},{(5,v(b_5))}}](0,5){p0}(2,-2){p2}(3,-3){p3}(5,0){p5}
      \pstGeonode[PointSymbol={x,x},PointName={{(1,v(b_1))},{(4,v(b_4))}}](1,3){p1}(4,1){p4}
      \psaxes[linewidth=.4pt,ticks=x,labels=x]{->}(0,0)(-.16,-3.5)(5.16,5.5)
      \pstMoveNode[PointSymbol=none,PointName=none](-2.1,0){p2}{o2}
      \pstMoveNode[PointSymbol=none,PointName=none](-3.1,0){p3}{o3}
      \pstLineAB[nodesepB=-.1,linestyle=dashed]{p2}{o2}
      \pstLineAB[nodesepB=-.1,linestyle=dashed]{p3}{o3}
      \pstLineAB[arrows=->]{o3}{o2}
      \naput{$v(\alpha)$}
    \end{pspicture}
    \caption{Newton polygon for a monic polynomial~\(g\) of degree~\(n=5\) with an isolated slope corresponding to the unique zero~\(\alpha\) of~\(g\) such that \(v(\alpha)=v(b_2)-v(b_3)\).}
    \label{fig:np}
  \end{figure}
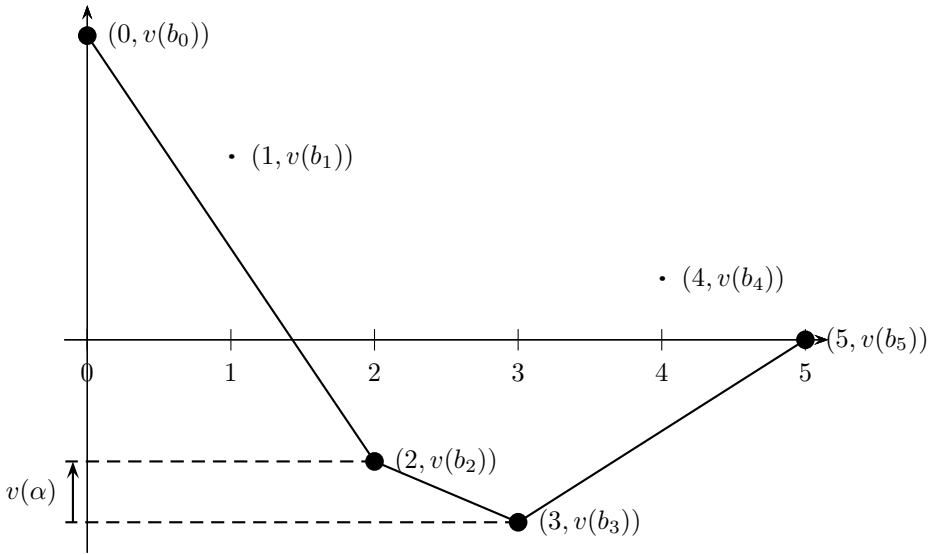
  \begin{enumerate}
  \item Then the unique zero~\(\alpha\) of~\(g\) in~\(\bV^{\mathrm{ac}}\) such that \(v(\alpha)=v(b_k)-v(b_{k+1})\) admits \(\frac{-b_k}{b_{k+1}}\) as \textup{immediate description} in the sense that \(\alpha=\frac{-b_k}{b_{k+1}}(1+\mu)\) with \(\mu\in\im[\bV^{\mathrm{ac}}]\).
  \item Furthermore, \(\nu=1+\mu\) is an henselian zero above~\(1\) for the polynomial \(h\in\bV[Y]\) defined by the formal equality \(b_k^{k+1}h(Y)=b_{k+1}^kg\bigl(\frac{-b_k}{b_{k+1}}Y\bigr)\) in which \(Y\)\kern-1.29214pt, \(b_k\), and~\(b_{k+1}\) are treated as indeterminates.
\item If we let \(f=a_0+a_{1}X+\dots+a_nX^n=h(1+X)\), then \(v(a_0)>0\) and \(v(a_1)=0\), and \(\mu\) is an henselian zero for~\(f\). If \(a_0=0\), then \(\mu=0\). If \(a_0\ne0\), we define \(t\) by the formal equality \(a_0t(X)=X^nf(\frac{-a_0}{a_1X})\) in which \(X,a_0,a_1\) are treated as indeterminates, and get that \(t\)~is a special polynomial, so that all zeros of~\(t\) except the special zero, corresponding to the zero~\(\alpha\) of~\(g\), are in \(\im[\bV^{\mathrm{ac}}]\). In a nutshell, a zero~\(\alpha\) corresponding to an isolated slope of a Newton polygon can always be explicited either as an element of~\(\bK\) or in the form \((a\beta+b)/(c\beta+d)\) with \(\beta\) the special zero of a special polynomial~\(t\), \(a,b,c,d\in\bV\), \(c\beta+d\ne0\), \(ad-bc\ne0\).
  \end{enumerate}

\end{proposition}
\begin{proof}[\textsl{Sketch of the proof of Theorem~\ref{hensel-dvf} in \cite{KL00}}]
  As stated in Theorem~\ref{acvf}, the formal theory \(\mathbf{T}_{\mathrm{acvf}}(\bK,\bV)\) is consistent. One can therefore use the Newton polygon with hypothetically added zeros of polynomials in~\(\bV[X]\) as follows.

  As shown in Proposition~\ref{newton}, to add the henselian zero for the Nagata polynomial~\(f\) is equivalent to adding the special zero~\(\beta\) for a special polynomial~\(t\).
  
  The extension~\((\bK[\alpha],\bValpha)\) is completely determined through Newton polygons because they provide a consistent answer to the two following questions for any polynomial~\(Q\in\bK[X]\). (1)~Is \(Q(\beta)=0\)? (2)~If \(Q(\beta)\ne0\), which value the valuation on~\(\bK[\alpha]\) should assign to~\(Q(\beta)\) in the value group? In fact, \cite[proposition~2.3]{KL00} shows how to compute in a uniform way these answers, and Proposition~\ref{newton} provides the opposite of this value as an isolated slope of a Newton polygon. The last assertion holds because every element of~\(\bK[\alpha]\) has an immediate description.
\end{proof}
Note that the morphism \(\pi_\alpha\colon\bK[x]\to\bK[\alpha]\), where \(\bK[x]=\bK[X]/\gen f\), is in general not injective. Putting Theorems~\ref{hensel-rdlr} and~\ref{hensel-dvf} together yields that there is a unique local morphism \(\bVf\to\bValpha\). 

\section{The theorem of de Felipe and Teissier}

Let \((\bR,\im[\bR])\) be a residually discrete local ring with a \textsl{minimal valuation}, i.e.\ a map \(v\colon\bR\to\vgi\) with \(\vg\) a linearly ordered discrete abelian group and \(\vgi=\vg\cup\{\infty\}\) such that
\begin{itemize}
\item \(v(a)\ge0\),
\item \(v(ab)=v(a)+v(b)\),
\item \(v(a)=0\) if and only if \(a\in\bRtimes\), \(v(a)>0\) if and only if \(a\in\im[\bR]\),
\item \(v(a+b)\ge\min(v(a),v(b))\),
\item \(v(a)<\infty\) or there is \(b\in\bR\) with \(v(b)<\infty\) and \(ba\) nilpotent (so that \(v(a)=\infty\)).
\end{itemize}
The last condition is the \textsl{minimality condition}; it says that \(\fp=\{\,a\in\bR:v(a)=\infty\}\) is a minimal prime ideal because a ring is a zero-dimensional local ring if and only if each of its elements is invertible or nilpotent (\cite[Lemma~IV.8.4]{CACM}).  

Let \(\bK\) be the field of fractions of \(\bR/\fp\).  The minimal valuation~\(v\) defines a valuation on~\(\bK\), still denoted by~\(v\), such that the morphism \(\bR/\fp\to\bV\) to the valuation ring~\(\bV=\{\,x\in\bK:v(x)\ge0\}\) is local. Let \(\theta\) be the composition \(\bR\to\bR/\fp\to\bK\) and let \(\bS=\bR\setminus\fp=\{\,c\in\bR:\theta(c)\ne0\,\}\). 

Let \(f\in\bR[X]\) be a Nagata polynomial. Let \(\bR[x]=\bR[X]/\gen f\) and \(\bR_f\) be as in Theorem~\ref{hensel-rdlr}, and consider \(\pi_f\colon\bR[x]\to\bR_f\).

Considering \(f\) as a Nagata polynomial in~\(\bV[X]\), let \((\bK[\alpha],\bValpha)\) be as in Theorem~\ref{hensel-dvf}.
Consider \(\pi_\alpha\colon\bK[x]\to\bK[\alpha]\), where \(\bK[x]=\bK[X]/\gen f\).

Let \(\theta[x]\) be the morphism \(\bR[x]\to\bK[x]\) corresponding to~\(\theta\).
Let \(\theta_f\) be the unique local morphism of local \(\bR\)-algebras such that \(\theta_f\circ\pi_f=\pi_\alpha\circ\theta[x]\).
  \[  \xymatrix{
    \bR\ar[d]\ar@<-1.3ex>@/_/[ddd]_\theta\ar[r]&\bR[x]\ar@/_/[ddd]_{\theta[x]}\ar[r]^{\pi_f}&\bR_f\ar[dd]\ar@<-.5ex>@/_/[ddd]_{\theta_f}\\
    \bR/\fp\ar[d]&&\\
    \bV\ar[d]\ar[rr]&&\bValpha\ar[d]\\
    \bK\ar[r]&\bK[x]\ar[r]^{\pi_\alpha}&\bK[\alpha]}
  \]

\begin{theorem}\label{th-step}
  The kernel of~\(\theta_f\) is a minimal prime ideal of~\(\bR_f\).
\end{theorem}

\begin{proof}
  The kernel~\(\fp[f]\) of~\(\theta_f\) is a detachable prime ideal because~\(\bK[\alpha]\) is a discrete field. Let \(\bS_f=\bR_f\setminus\fp[f]=\{\,\gamma\in\bR_f:\theta_f(\gamma)\ne0\,\}\). We have to prove, for~\(\gamma\in\bR_f\), that \(\gamma\in \bS_f\) or there is a \(\zeta\in \bS_f\) such that \(\zeta\gamma\) is nilpotent. As \(\bR_f\) is a localisation of \(\bR[x]\), we have to prove this only for \(\gamma=\pi_f(q(x))\) with \(q\in\bR[X]\). Let \(q_1=\theta[X](q)\) and \(\delta=\pi_\alpha(q_1(x))=\theta_f(\gamma)\).
  \[  \xymatrix{
    \bR\ar[d]_\theta\ar[r]&\bR[x]\ar[d]_{\theta[x]}\ar[r]^{\pi_f}&\bR_f\ar[d]_{\theta_f}\\
    \bK\ar[r]&\bK[x]\ar[r]^{\pi_\alpha}&\bK[\alpha]}\qquad
  \xymatrix{
    q(x)\ar@{|->}[d]_{\theta[x]}\ar@{|->}[r]^{\pi_f}&\gamma\ar@{|->}[d]_{\theta_f}\\
    q_1(x)\ar@{|->}[r]^{\pi_\alpha}&\delta}
  \]
  Compute the characteristic polynomial~\(g(T)\) of~\(q(x)\) in~\(\bR[x]\). The same computation yields for~\(q_1(x)\) in~\(\bK[x]\) the characteristic polynomial~\(g_1=\theta[T](g)\): one has
  \[
    \begin{aligned}
      g(T)&=T^n+a_{n-1}T^{n-1}+\dots+a_0&\text{and}&&g(q)&=0\text;\\
      g_1(T)&=T^n+\theta(a_{n-1})T^{n-1}+\dots+\theta(a_0)&\text{and}&&g_1(q_1)&=0\text.\\
    \end{aligned}
  \]
  In particular, \(g(\gamma)=0\) in~\(\bR_f\) and \(g_1(\delta)=0\) in~\(\bK[\alpha]\).

  As \(\bK\)~is discrete, we have either \(g_1(0)=0\) or \(g_1(0)\ne0\) in~\(\bK\).
  \begin{itemize}
  \item If \(g_1(0)\ne0\), then, as \(\delta=0\implies g_1(0)=g_1(\delta)=0\), we have \(\delta\ne0\); as \(\delta=\theta_f(\gamma)\), \(\gamma\in \bS_f\), so we are done.
  \item If \(g_1(0)=0\), let us write \(g_1(T)=T^kh_1(T)\) with \(h_1(0)\ne0\) and \(0<k\le n\), so that \(\theta(a_j)=0\) for \(0\le j<k\). By minimality of~\(\fp\), \(\theta(a_j)\ne0\) being impossible, there is a \(b_j\) such that \(b_ja_j\) is nilpotent. Let \(b=b_0\cdots b_{k-1}\in \bS\); if we write \(g(T)=T^kh(T)+a(T)\), there is~\(N\) such that \((ba(T))^N=0\); note that \(\theta[T](h)=h_1\).

    As \(g_1(\delta)=\delta^kh_1(\delta)=0\) and \(\bK[\alpha]\) is discrete, we have either \(\delta=0\) or \(h_1(\delta)=0\).
    \begin{itemize}
    \item If \(h_1(\delta)=0\), then, as \(\delta=0\implies h_1(0)=h_1(\delta)=0\), we have \(\delta\ne0\), i.e.\ \(\gamma\in \bS_f\), so we are done.
    \item If \(\delta=0\), let \(\epsilon=h(\gamma)\): then \(\theta_f(\epsilon)=h_1(\delta)=h_1(0)\ne0\), so that \(\epsilon\in \bS_f\). As \(g(\gamma)=0\), \(bg(\gamma)=b\gamma^k\epsilon+ba(\gamma)=0\); thus \((b\gamma^k\epsilon)^N=(-ba(\gamma))^N=0\). For \(\zeta=b\epsilon\) we have therefore \(\zeta\in \bS_f\) and \(\zeta\gamma\) nilpotent, so we are done.
    \end{itemize}
  \end{itemize}
  Thus we are done in all cases.
\end{proof}

\begin{remark}\label{rem-reduit}
  For the purpose of proving the result of de Felipe and Teissier, it would have sufficed to consider the case where \(\bR\) is reduced, for which the proof does not resort to nilpotents, because this property is inherited by~\(\bR_f\).
\end{remark}

\begin{remark}\label{rem-single}
  As every extension \((\bK[\alpha_1,\dots,\alpha_n],\bV_{\!\alpha_1,\dots,\alpha_n})\) obtained by adding step by step zeros of Nagata polynomials may be obtained in a single step, it would even have sufficed to consider the case where \(\bR\) is a domain. However, the proof that every extension may be obtained in a single step is quite involved; a constructive proof using the multivariate Hensel lemma is given in~\cite{ACL2014}.
\end{remark}

\end{document}